\documentclass[11pt,oneside,leqno]{amsart}
\usepackage{amsxtra}
\usepackage{amsopn}
\usepackage{amsmath,amsthm,amssymb}
\usepackage{amscd}
\usepackage{amsfonts}
\usepackage{latexsym}
\usepackage{verbatim}
\usepackage{color}

\theoremstyle{plain}
\newtheorem{theorem}{Theorem}[section]

\newtheorem*{question*}{Question 1}
\newtheorem*{question**}{Question 2}
\newtheorem{definition}[theorem]{Definition}

\newtheorem{prop}[theorem]{Proposition}
\newtheorem{cor}[theorem]{Corollary}
\newtheorem{rem}[theorem]{Remark}
\newtheorem{ex}[theorem]{Example}

\newcommand\C{{\mathbb C}}

\newcommand{\del}{\partial}
\newcommand{\delbar}{\overline{\del}}

\begin{document}
\title[An integral condition involving $\delbar$-harmonic $(0,1)$-forms ]{An integral condition involving $\delbar$-harmonic $(0,1)$-forms}

\author{Anna Fino}
\address[Anna Fino]{Dipartimento di Matematica ``G. Peano'' \\
Universit\`{a} degli studi di Torino \\
Via Carlo Alberto 10\\
10123 Torino, Italy}
\address{Department of Mathematics and Statistics, Florida\\
International University Miami Florida\\
33199, USA
}
\email{annamaria.fino@unito.it, afino@fiu.edu}

\author{Nicoletta Tardini}
\address[Nicoletta Tardini]{Dipartimento di Scienze Matematiche, Fisiche e Informatiche\\
Unit\`a di Matematica e Informatica\\
Universit\`a degli Studi di Parma\\
Parco Area delle Scienze 53/A\\
43124 Parma, Italy}
\email{nicoletta.tardini@unipr.it}

\author{Adriano Tomassini}
\address[Adriano Tomassini]{Dipartimento di Scienze Matematiche, Fisiche e Informatiche\\
Unit\`a di Matematica e Informatica\\
Universit\`a degli Studi di Parma\\
Parco Area delle Scienze 53/A\\
43124 Parma, Italy}
\email{adriano.tomassini@unipr.it}


\begin{abstract} 
We  study  compact almost complex manifolds admitting a  Hermitian metric   satisfying  an integral condition involving $\overline \partial$-harmonic $(0,1)$-forms.
We prove that   this  integral condition is  automatically satisfied, if the Hermitian metric  on the compact almost complex manifold is strongly Gauduchon. Under the further assumption that
 the almost complex structure is integrable,   we show that the integral condition for  a Gauduchon metric is equivalent to be  strongly Gauduchon.

 In particular, a compact complex surface with  a Gauduchon metric  satisfying   the integral condition is automatically K\"ahler.
If we   drop the integrability assumption on the complex structure,  we show   that there exists a compact almost complex 4-dimensional
manifold   with  a  Hermitian metric  satisfying
the integral condition, but which does not admit any compatible almost-K\"ahler metric. 
\end{abstract}
\maketitle
\section{Introduction}
Let $(M,J)$ be a compact complex manifold. The existence of a K\"ahler metric $g$ on $(M,J)$ gives rise to a cohomological decomposition of the de Rham cohomology as the direct sum of Dolbeault cohomology groups and all the complex cohomologies, Bott-Chern, Aeppli and Dolbeault coincide. Moreover, the existence of a K\"ahler metric imposes also strong restrictions on the topology of $M$. Furthermore, the $\del\delbar$-Lemma holds on $(M,J)$. For a given almost complex manifold $(M,J)$, with non-integrable $J$, such a cohomological decomposition does not hold anymore and Dolbeault, Bott-Chern and Aeppli cohomologies cannot be defined in the usual way, since $\delbar^2\neq 0$; nevertheless, once a Hermitian metric $g$ is fixed on $(M,J)$, the Dolbeault Laplacian can still be defined as in the complex case and it turns out that it is a second order elliptic differential operator, therefore, for every $(p,q)$, the kernels of its restriction to the spaces of $(p,q)$-forms on $(M,J)$ are finite dimensional complex vector spaces, called $\delbar$-{\em harmonic spaces}. Kodaira and Spencer (see \cite[Problem 20]{hirzebruch}) asked whether the dimension of such spaces could depend on the Hermitian metric $g$. Very recently Holt and Zhang in \cite{holt-zhang} answered such a question, providing a family of almost complex structures on the Kodaira-Thurston manifold such that the Hodge number $h^{0,1}_{\delbar}$ varies with different choices of Hermitian metrics. For other results on the study of the $\delbar$-harmonic spaces see \cite{tardini-tomassini-dim4}, \cite{piovani} and \cite{tardini-tomassini-dim6}.  \newline 
	The aim of this paper is the study of Hermitian metrics on a compact $2n$-dimensional almost complex manifold satisfying an integral  condition involving the space of $\delbar$-harmonic $(0,1)$-forms. More precisely,
let $(M,J)$ be a $2n$-dimensional almost complex manifold. A Hermitian metric $g$  on $(M,J)$ is called \emph{Gauduchon} if $\del\delbar\omega^{n-1}=0$, where by $\omega$ we denote its fundamental form. If $M$ is compact these metrics always exist, in fact Gauduchon in \cite{gauduchon} proved that if
 $(M,J)$ is a $2n$-dimensional compact almost complex manifold with $n>1$, then any conformal class of any given Hermitian metric contains a Gauduchon metric, unique up to multiplication with positive constants. If  $n =2$  and $J$ is integrable,  a  Gauduchon metric  is SKT (or pluriclosed) and so one can be found in the conformal class of any given Hermitian metric.  For $n >2$ the theory  about  SKT metrics is very different, see  for instance \cite{FPS, GGP, Sw, FP} and the references therein  for  explicit  constructions.
 A subclass of  Gauduchon  metrics  on complex manifolds was defined by Popovici in \cite{popovici-surface} as \emph{strongly Gauduchon}, namely Hermitian metrics  $\omega$ such that $\del\omega^{n-1}$ is $\delbar$-exact. 
Examples of strongly Gauduchon metrics are \emph{balanced metrics} namely Hermitian metrics $\omega$ satisfying $d\omega^{n-1}=0$. These metrics play an important role in complex geometry and they have been widely studied (see for instance \cite{michelson, fino-vezzoni-1, fino-vezzoni-skt-bal, andrada-villacampa} and the references therein).
If $J$ is non-integrable $\delbar^2\neq0$, however the definition of strongly Gauduchon metric can be naturally generalized asking that $\omega$ is Gauduchon and it satisfies  the condition that $\del\omega^{n-1}$ is $\delbar$-exact. We study these metrics on compact almost complex manifolds and we show that they satisfy an integral condition involving $\overline \partial$-harmonic $(0,1)$-forms (such condition can be also thought as an orthogonality condition by Proposition \ref{prop:orthogonal}). More precisely, in Theorem  \ref{thm:sg-implica-ci} we prove that
on compact  $2n$-dimensional almost complex manifolds every  strongly Gauduchon metric $\omega$ satisfies
$$
\int_M\del\eta\wedge\omega^{n-1}=0,
$$
for every  $\eta\in \mathcal{H}^{0,1}_{\delbar}(M)$. Here $\mathcal{H}^{0,1}_{\delbar}(M)$ denotes the space of $\delbar$-harmonic $(0,1)$-forms. This space is well defined and finite dimensional but notice that if $J$ is non-integrable it does not have a cohomological meaning and its dimension depends on the choice of the Hermitian metric $\omega$ (see \cite{holt-zhang}).
If $J$ is integrable, Gauduchon metrics satisfy the integral condition and such condition does not depend on harmonic forms but just on the Dolbeault cohomology classes (see Theorems \ref{prop:sg-ci}, \ref{thm:sg-ci-cohomology})
\begin{theorem}
Let $M$ be a compact complex manifold  of complex dimension $n$ and let $\omega$ be a Gauduchon metric, then $\omega$ is strongly Gauduchon if and only if
$$
\int_M\del\left[\eta\right]\wedge\omega^{n-1}=0,
$$
for every  $\left[\eta\right]\in H^{0,1}_{\delbar}(M)$.
\end{theorem}
Recall that strongly Gauduchon metrics coincide with Gauduchon metrics when $(M,J)$ is compact and satisfies the $\del\delbar$-lemma. 
In fact, compact complex manifolds such that every Gauduchon metric is strongly Gauduchon are called sGG and they were introduced and studied in \cite{popovici-ugarte}. In Proposition \ref{prop:sgg-ci} and Corollary \ref{cor:sgg-ci} we observe that if either every Hermitian metric satisfies the integral condition or if $h^{0,1}_{\delbar}=0$ then the manifold is sGG.\\
Notice that as recalled in Proposition \ref{prop:left-inv-gauduchon}
on compact quotients  of  connected, simply connected Lie groups by  lattices  endowed with an invariant almost complex structure every invariant Hermitian metric $\omega$ is Gauduchon.  Here by  invariant structure we mean a structure induced by a left-invariant one.  Hence, in order to construct  on solvmanifolds  Hermitian metrics satisfying the integral condition but  not strongly Gauduchon  we  consider non invariant Hermitian metrics. More precisely,  in Example \ref{ex:iwasawa} we construct on the Iwasawa manifold a family of non invariant Hermitian metrics that are not Gauduchon but they satisfy the integral condition.  \\
One should also notice that the existence of a strongly Gauduchon metric is quite restrictive on compact complex surfaces, because in this case it implies the existence of a K\"ahler metric.   
We ask the following
%
 \begin{question*} \label{conj:surface}
Let $(M,J)$ be a compact complex surface. If there exists a Hermitian metric satisfying the integral condition then does $(M, J)$  admit a K\"ahler metric?
\end{question*}

We show that   the answer is negative if we do not assume that $J$ is integrable. Indeed, in Section \ref{sect:ci-surf} we construct 
a compact almost complex $4$-dimensional manifold  that admits a strongly Gauduchon metric but it does not admit any almost-K\"ahler metric. In particular, such a metric satisfies the integral condition.
This example is constructed starting with the  4-dimensional filiform nilmanifold and it shows
that the result in \cite{popovici-surface} saying that strongly Gauduchon compact complex surfaces are K\"ahler does not hold as soon as we drop the integrability assumption on the complex structure.

\smallskip

{\it Acknowledgements.} 
The authors would like to thank Scott O. Wilson for pointing out a mistake in a Remark that doesn't affect the results.
The first   author is partially  supported by Project PRIN 2017 “Real and complex manifolds: Topology, Geometry and Holomorphic Dynamics”, by GNSAGA of INdAM and by  a grant from the Simons Foundation (\#944448).  The second author has financially been supported by the Programme ``FIL-Quota Incentivante'' of University of Parma and co-sponsored by Fondazione Cariparma and by GNSAGA of INdAM.  The third author is  supported by Project PRIN 2017 “Real and complex manifolds: Topology, Geometry and Holomorphic Dynamics” and by GNSAGA of INdAM. The authors would like also to thank the anonymous referee for useful comments.

\section{Preliminaries}

In this Section we will recall some well known facts about almost complex manifolds and Hermitian metrics. Let $M$ be a $2n$-dimensional smooth manifold and let $J$ be an almost complex structure on $M$, namely $J$ is a $(1,1)$-tensor such that $J^2=-\text{Id}$. Then $J$ induces naturally on the space of differential forms $A^\bullet(M)$ a bigrading,
$$
A^\bullet(M)=\bigoplus_{p+q=\bullet}A^{p,q}(M)\,.
$$
According to this decomposition, the exterior derivative $d$ splits into four operators
$$
d:A^{p,q}(M)\to A^{p+2,q-1}(M)\oplus A^{p+1,q}(M)\oplus A^{p,q+1}(M)\oplus A^{p-1,q+2}(M),
$$
$$
d=\mu+\del+\delbar+\bar\mu\,,
$$
where $\mu$ and $\bar\mu$ are differential operators that are linear over functions.
In particular, $J$ is integrable, that is $J$ induces a complex structure on $M$, if and only if $\mu=\bar\mu=0$.\\
In general, the condition $d^2=0$ gives the following
$$
\left\lbrace
\begin{array}{lcl}
\mu^2 & =& 0\\
\mu\del+\del\mu & = & 0\\
\del^2+\mu\delbar+\delbar\mu & = & 0\\
\del\delbar+\delbar\del+\mu\bar\mu+\bar\mu\mu & = & 0\\
\delbar^2+\bar\mu\del+\del\bar\mu & = & 0\\
\bar\mu\delbar+\delbar\bar\mu & = & 0\\
\bar\mu^2 & =& 0
\end{array}
\right.\,.
$$
In particular, on almost complex manifolds  we have $\delbar^2\neq 0$ and so the Dolbeault cohomology of $M$
$$
H^{\bullet,\bullet}_{\delbar}(M):=
\frac{\text{Ker}\,\delbar}{\text{Im}\,\delbar}
$$
is well defined if and only if $J$ is integrable.\\
A Hermitian metric $g$ on an almost complex manifold $(M,J)$ is a Riemannian metric such that $g(u,v)=g(Ju,Jv)$ for every vector fields $u,v\in\Gamma(TM)$. The $2$-form $\omega$ defined by
\begin{equation*}
\omega(u,v)=g(Ju,v), \ \ \forall u,v\in\Gamma(TM),
\end{equation*}
is called the fundamental form of $g$. We will call $(M,J,\omega)$ an almost Hermitian manifold. If in addition, $\omega$ is $d$-closed, namely $d\omega=0$, then $(M,J,\omega)$ is called {\em almost K\"ahler manifold}.\\
Let $(M,J,\omega)$ be an almost Hermitian manifold and denote with $*$ the 
 associated $\mathbb{C}$-linear Hodge-$*$-operator, then the following operators
\begin{equation*}
d^*=-*d*,\ \ \ \mu^*=-*\bar\mu*,\ \ \ \del^*=-*\delbar*,\ \ \ \delbar^*=-*\del*,\ \ \ \bar\mu^*=-*\mu*,
\end{equation*}
are the formal adjoint operators respectively of $d,\mu,\del,\delbar,\bar\mu$. Recall that $\Delta_{d}=dd^*+d^*d$ is the Hodge Laplacian, and, as in the integrable case, set 
\begin{equation*}
\Delta_{\del}=\del\del^*+\del^*\del,\ \ \ \Delta_{\delbar}=\delbar\delbar^*+\delbar^*\delbar,
\end{equation*}
respectively as the $\del$ and $\delbar$ Laplacians. 
We will be mainly interested in $\Delta_{\delbar}$, 
this is a second order, elliptic, differential operator and we will denote its kernel by
$$
\mathcal{H}^{p,q}_{\delbar}(M):=\text{Ker}\,\Delta_{\delbar_{\vert A^{p,q}(M)}}\,.
$$
In particular, if $M$ is compact these spaces are finite-dimensional and their dimensions will be denoted by $h^{p,q}_{\delbar}$ (cf. \cite{hirzebruch}). The numbers $h^{p,q}_{\delbar}$ are called Hodge numbers.
Moreover, if $M$ is compact, then $\alpha\in \mathcal{H}^{p,q}_{\delbar}(M)$ if and only if
$$
\delbar\alpha=0\quad\text{and}\quad \del*\alpha=0.
$$
By \cite{holt-zhang} (see also \cite{tardini-tomassini-dim4}) we know that, if $J$ is non-integrable, the dimensions $h^{p,q}_{\delbar}$ of the spaces of harmonic forms  are not holomorphic invariants, more precisely they depend on the choice of the Hermitian metric. Clearly, if $(M,J)$ is a complex manifold,  the Hodge numbers $h^{p,q}_{\delbar}$ are nothing but the dimensions of the associated Dolbeault cohomology groups $H^{p,q}_{\delbar}(M)$.\\
This follows from the usual Hodge decomposition
$$
A^{p,q}(M)=\mathcal{H}^{p,q}_{\delbar}(M)
\stackrel{\perp}{\oplus}
\text{Im}\,\delbar\stackrel{\perp}{\oplus}\text{Im}\,\delbar^*\,.
$$
We point out that, when $J$ is non-integrable, the Hodge decomposition holds but it is not an orthogonal decomposition. More precisely, from
\cite[Theorem 4.1]{cirici-wilson}, one has
$$
A^{p,q}(M)=\mathcal{H}^{p,q}_{\delbar}(M)
\stackrel{\perp}{\oplus}
\left(\text{Im}\,\delbar+\text{Im}\,\delbar^*\right)\,.
$$

\section{The integral Condition}

In this Section we define and study the following class of Hermitian metrics on compact almost complex manifolds.

\begin{definition}
Let $(M,J)$ be a compact almost complex manifold of dimension $2n$. A Hermitian metric $\omega$ on $(M,J)$ is said to satisfy the \emph{integral condition} if
\begin{equation}\label{eq:ci}
\int_M\del\eta\wedge\omega^{n-1}=0,
\end{equation}
for every  $\eta\in \mathcal{H}^{0,1}_{\delbar}(M)$.
\end{definition}

Notice that, since $M$ is compact, as mentioned the spaces $\mathcal{H}^{\bullet,\bullet}_{\delbar}(M)$ of $\delbar$-harmonic forms are well-defined and finite dimensional even though $J$ is non-integrable (cf. \cite{hirzebruch}).
However, the dimension of $\mathcal{H}^{0,1}_{\delbar}(M)$ depends in general on the Hermitian metric $\omega$.

\begin{rem}
Clearly, if $\omega$ is such that $\mathcal{H}^{0,1}_{\delbar}(M)=\mathcal{H}^{0,1}_{\del}(M)$,  then the integral condition is trivially satisfied.
\end{rem}

We give now a characterization of the Hermitian metrics satisfying the integral condition.

\begin{prop}\label{prop:orthogonal}
Let $(M,J)$ be a  compact $2n$-dimensional  almost complex manifold. A Hermitian metric $\omega$ on $(M,J)$ satisfies the integral condition if and only if $\del^*\omega\perp \mathcal{H}^{0,1}_{\delbar}(M)$.
\end{prop}

\begin{proof}
Notice that the integral condition (\ref{eq:ci}) can be rewritten, for every $\eta\in \mathcal{H}^{0,1}_{\delbar}(M)$, as
$$
0=\int_M\del\eta\wedge\omega^{n-1}=
(n-1)!\int_M\del\eta\wedge *\omega=
(n-1)!\left(\del\eta,\omega\right)_{L^2}=(n-1)!\left(\eta,\del^*\omega\right)_{L^2}
$$
namely (\ref{eq:ci}) is equivalent to  $\del^*\omega\perp \mathcal{H}^{0,1}_{\delbar}(M)$. 
\end{proof}


Recall that a Hermitian metric $\omega$ on a $2n$-dimensional almost complex manifold is called \emph{Gauduchon} if $\del\delbar\omega^{n-1}=0$. 
The \emph{Lee form} of a Hermitian metric $\omega$ is defined as the unique $1$-form $\theta$ such that $d\omega^{n-1}=\theta\wedge\omega^{n-1}$. Then, a Hermitian metric $\omega$ is Gauduchon if and only if $d^*\theta=0$.\\
If $n>1$ by the celebrated result in \cite{gauduchon} (see also \cite{gauduchon1}), a Gauduchon metric always exists in every conformal class of a Hermitian metric and it is unique up to multiplication with positive constants.
On complex manifolds a special class of Gauduchon metrics is given by the so called  \emph{strongly Gauduchon} metrics, namely Hermitian metrics $\omega$ satisfying $\del\omega^{n-1}$ being $\delbar$-exact (see \cite{popovici-surface} and the references therein). Notice that this last definition originally given for complex manifolds can be naturally extended on almost complex manifolds.
Indeed, on almost complex manifolds $\delbar^2\neq 0$ and $\del\delbar+\delbar\del\neq 0$, but for bidegree reasons $$\del\delbar\omega^{n-1}=-\delbar\del\omega^{n-1}.$$
Moreover, if
$$
\del\omega^{n-1}=\delbar\lambda
$$
where $\lambda\in A^{n,n-2}(M)$. Then
$$
\delbar\del\omega^{n-1}=\delbar^2\lambda=-\del\bar\mu\lambda
$$
and this last term is not zero in general. So, on almost complex manifolds, in the definition of strongly Gauduchon metrics we need to assume in addition that $\omega$ is Gauduchon. Namely, we will say that a Hermitian metric $\omega$ on an almost complex manifold is \emph{strongly Gauduchon} if it is Gauduchon and $\del\omega^{n-1}$ is $\delbar$-exact.  
Then we prove the following

\begin{theorem}\label{thm:sg-implica-ci}
Let $(M,J)$ be a compact  $2n$-dimensional almost complex manifold and let $\omega$ be a  strongly Gauduchon metric. Then,
$$
\int_M\del\eta\wedge\omega^{n-1}=0,
$$
for every  $\eta\in \mathcal{H}^{0,1}_{\delbar}(M)$.
\end{theorem}
\begin{proof}
Suppose that $\omega$ is strongly Gauduchon, namely $\omega$ is Gauduchon and $\del\omega^{n-1}=\delbar\lambda$ for some $\lambda\in A^{n,n-2}(M)$. Then the integral condition (\ref{eq:ci}) is obtained as a consequence of Stokes Theorem. Indeed,
for every  $\eta\in \mathcal{H}^{0,1}_{\delbar}(M)$,
$$
\begin{array}{lll}
\int_M\del\eta\wedge\omega^{n-1}&=&
\int_M d\eta\wedge\omega^{n-1}=
\int_M\eta\wedge d\omega^{n-1}\\[10pt]
&=& \int_M\eta\wedge\del\omega^{n-1}
=\int_M\eta\wedge\delbar\lambda=
\int_M\eta\wedge d\lambda\\[10pt]
&=&\int_Md\eta\wedge\lambda=
\int_M\delbar\eta\wedge\lambda=0
\end{array}
$$
since $\delbar\eta=0$.\\
\end{proof}

If the almost complex structure is integrable then also the converse is true.

\begin{theorem}\label{prop:sg-ci}
Let $(M,J)$ be a compact complex manifold and let $\omega$ be a Gauduchon metric, then $\omega$ is strongly Gauduchon if and only if
$$
\int_M\del\eta\wedge\omega^{n-1}=0,
$$
for every  $\eta\in \mathcal{H}^{0,1}_{\delbar}(M)$.
\end{theorem}
\begin{proof}
The first implication follows from Theorem \ref{thm:sg-implica-ci}.
We then need to prove the converse.\\
By Proposition \ref{prop:orthogonal} we have
$$
0=\int_M\del\eta\wedge\omega^{n-1}=(n-1)!\int_M\del\eta\wedge*\omega
=(n-1)!(\del\eta,\omega)_{L^2}=(n-1)!(\eta,\del^*\omega)_{L^2}
$$
for every $\eta\in\mathcal{H}^{0,1}_{\delbar}(M)$.\\
Hence $\del^*\omega\perp\mathcal{H}^{0,1}_{\delbar}(M)$,
therefore, 
$$
\del^*\omega\in\text{Im}\,\delbar\oplus\text{Im}\,\delbar^*=
\delbar A^{0,0}(M)\oplus\delbar^*A^{0,2}(M),
$$
namely, there exists $f\in\mathcal{C}^{\infty}(M)$ and $\beta\in A^{0,2}(M)$ such that
$$
\del^*\omega=\delbar f+\delbar^*\beta.
$$
This is equivalent to 
(using the $\mathbb{C}$-linear Hodge-$*$-operator)
$$
\delbar*\omega=*\delbar f+\del*\beta\qquad\iff\qquad
\delbar\frac{\omega^{n-1}}{(n-1)!}=-\del^*(*f)+\del*\beta.
$$
Since $\omega$ is Gauduchon, that is $\del\delbar\omega^{n-1}=0$,  we get
$$
\del\del^*(*f)=0
$$
hence $\del^*(*f)=0$ and so
$$
\delbar\omega^{n-1}=\del\left( (n-1)!*\beta\right)\in\text{Im}\,\del\,,
$$
or equivalently, since $\omega$ is real, $\del\omega^{n-1}\in\text{Im}\,\delbar$,
that means that $\omega$ is strongly Gauduchon.
\end{proof}

\begin{rem}
Notice that if $J$ is integrable, namely $(M,J)$ is a compact complex manifold then if $\omega$ is a Gauduchon metric on $M$ then, by Stokes' Theorem, the  integral condition (\ref{eq:ci}) does not depend on the representative in the cohomology classes of $H^{0,1}_{\delbar}(M)$, namely
$$
\int_M\del\left[\eta\right]\wedge\omega^{n-1}
$$
is well defined for every  $\left[\eta\right]\in H^{0,1}_{\delbar}(M)$. 
\end{rem}
Hence we can reformulate
\begin{theorem}\label{thm:sg-ci-cohomology}
Let $M$ be a compact complex manifold and let $\omega$ be a Gauduchon metric, then $\omega$ is strongly Gauduchon if and only if
$$
\int_M\del\left[\eta\right]\wedge\omega^{n-1}=0,
$$
for every  $\left[\eta\right]\in H^{0,1}_{\delbar}(M)$.
\end{theorem}
Recall that a compact complex manifold  is called \emph{sGG} if every Gauduchon metric is also strongly Gauduchon (see \cite{popovici-ugarte}). Then the following holds
\begin{prop}\label{prop:sgg-ci}
Let $M$ be a compact complex manifold of complex dimension $n$  and suppose that for every Hermitian metric $\omega$ on $M$,
$$
\int_M\del\eta\wedge\omega^{n-1}=0,
$$
for every  $\eta\in \mathcal{H}^{0,1}_{\delbar}(M)$, then $M$ is sGG.
\end{prop}
\begin{proof}
Suppose that, $$
\int_M\del\eta\wedge\omega^{n-1}=0,
$$
for every  $\eta\in \mathcal{H}^{0,1}_{\delbar}(M)$ and for every Hermitian metric $\omega$ on $M$. In particular, the integral condition is satisfied by all the Gauduchon metrics and so by  Theorem  \ref{prop:sg-ci}, they are all strongly Gauduchon and by definition this implies that $M$ is sGG.
\end{proof}
Moreover, as a corollary we obtain (cf. \cite[Theorem 1.4]{popovici-ugarte})
\begin{cor}\label{cor:sgg-ci}
Let $M$ be a compact complex manifold with $h^{0,1}_{\delbar}=0$. Then, $M$ admits a strongly Gauduchon metric.
\end{cor}
\begin{proof}
Suppose that $h^{0,1}_{\delbar}=0$ hence, the integral condition is trivially satisfied by every Hermitian metric. In particular, Gauduchon metrics, that always exist, satisfy such condition. This implies that every Gauduchon metric is strongly Gauduchon.
\end{proof}

Notice that if $M$ is a compact complex surface with $h^{0,1}_{\delbar}=0$, then $M$ admits a strongly Gauduchon metric and so by \cite{popovici-surface} $M$ admits a K\"ahler metric. This is a different way to prove the well known result that compact complex surfaces with $h^{0,1}_{\delbar}=0$ are K\"ahler (cf. e.g., \cite[Lemma 2.6, Theorem 2.7]{barth-hulek-peters-van de ven}).\\

We will now provide some examples of Hermitian metrics satisfying the integral condition on compact manifolds.
An important source of examples of (almost) complex manifolds admitting special Hermitian metrics is given by solvmanifolds, that are compact quotients of simply connected, connected solvable Lie groups by a lattice.\\
For this reason we prove the following

\begin{prop}\label{prop:left-inv-gauduchon}
Let $M$ be the compact quotient of a connected, simply connected Lie group by a lattice. Suppose that $M$ is endowed with an invariant almost complex structure $J$. Then, any invariant Hermitian metric $\omega$ is Gauduchon. 
\end{prop}

\begin{proof}
Let $\omega$ be an invariant Hermitian metric on $M$ and let
 $\theta$ be the associated Lee form. Then, since $\omega$ and $J$ are invariant, $d^*\theta$ is an invariant  function on $M$ and so it is constant.
However, if we compute
$$
\int_M d^*\theta\frac{\omega^n}{n!}=\left(d^*\theta,1\right)_{L^2}=\left(\theta,d\,(1)\right)_{L^2}=0
$$
and so $d^*\theta=0$, namely $\omega$ is a Gauduchon metric.\\
\end{proof}

As a consequence of Proposition \ref{prop:left-inv-gauduchon} and Theorem \ref{prop:sg-ci}, in order to construct new examples of Hermitian metrics satisfying (\ref{eq:ci}) that are not strongly Gauduchon on solvmanifolds we will need to consider non invariant Hermitian metrics.

\begin{ex}\label{ex:iwasawa}
Let $\mathbb{I}$ be the Iwasawa manifold defined as the  compact quotient $\mathbb{I}:=\Gamma\backslash\mathbb{H}_3$,  where
$$
\mathbb{H}_3:=\left\lbrace
\left[\begin{matrix}
1 & z_1 &   z_3 \\
0 & 1   & z_2\\
0 & 0 & 1\\
\end{matrix}\right]
\mid z_1,z_2,z_3\in\mathbb{C}
\right\rbrace
$$
and
$$
\Gamma:=\left\lbrace
\left[\begin{matrix}
1 & \gamma_1 &   \gamma_3 \\
0 & 1   & \gamma_2\\
0 & 0 & 1\\
\end{matrix}\right]
\mid \gamma_1,\gamma_2,\gamma_3\in\mathbb{Z}[\,i\,]
\right\rbrace\,.
$$
Set
$$
\varphi^1:=dz_1,\qquad
\varphi^2:=dz_2,\qquad
\varphi^3:=dz_3-z_1dz_2
$$
as global co-frame of $(1,0)$-forms and let $V_1,V_2,V_3$ the dual frame of vector fields. We get
$$
V_1=\frac{\partial}{\partial z_1}\,,\quad 
V_2=\frac{\partial}{\partial z_2}+z_1\frac{\partial}{\partial z_3}\,,\quad 
V_3=\frac{\partial}{\partial z_3}\,.
$$\\
Then the complex structure equations become
$$
d\varphi^1=d\varphi^2=0, \qquad
d\varphi^3=-\varphi^1\wedge\varphi^2\,.
$$
We will set as usual $\varphi^{i\bar j}=\varphi^i\wedge\varphi^{\bar j}$ and so on.
Let $\sigma:\mathbb{I}\to\mathbb{R}$ be a smooth function such that
$$
V_3\bar V_3(\sigma)\neq 0,
$$
that is $\sigma$ can be identified to a smooth function $\sigma:\C^3\to \mathbb{R}$ satisfying the following conditions:
$$
\sigma(z_1+\gamma_1,z_2+\gamma_2,z_3+\gamma_1z_2+\gamma_3)=\sigma(z_1,z_2,z_3),\quad \forall (\gamma_1,\gamma_2,\gamma_3)\in\Gamma,
$$
$$
\frac{\partial^2 \sigma}{\partial z_3\partial \overline{z_3}}\neq 0
$$
 and consider the following Hermitian metric
$$
\omega_\sigma:=\frac{i}{2}\left(e^{-2\sigma}\varphi^{1\bar 1}+e^{-2\sigma}\varphi^{2\bar 2}+e^{2\sigma}\varphi^{3\bar 3}\right).
$$
Notice that $\omega_\sigma$ is not Gauduchon, indeed
$$
\begin{array}{lll}
\del\delbar\omega_\sigma^2&=&
-\frac{1}{2}\del\delbar\left(e^{-4\sigma}\varphi^{1\bar 12\bar2}+
\varphi^{1\bar 13\bar3}+\varphi^{2\bar 23\bar3}\right)=
2e^{-4\sigma}\del\delbar\sigma\wedge\varphi^{1\bar 12\bar 2}\\ [10pt]
&=& 2e^{-4\sigma}V_3\bar V_3(\sigma)\varphi^{1\bar 12\bar 2 3\bar3}\neq 0.
\end{array}
$$
Hence in particular $\omega_\sigma$ is not strongly Gauduchon.\\
Now observe that with respect to $\omega_\sigma$,
$$
\mathcal{H}^{0,1}_{\delbar}=\left\langle \bar\varphi^1,\bar\varphi^2\right\rangle.
$$
Indeed, denoting with $*_\sigma$ the $\mathbb{C}$-linear Hodge-star-operator with respect to $\omega_\sigma$, we get
$$
*_\sigma\bar\varphi^{1}=\varphi^{23\bar1\bar2\bar3},\quad
*_\sigma\bar\varphi^{2}=-\varphi^{13\bar1\bar2\bar3}
$$
and so
$$
\del*_\sigma\bar\varphi^{1}=0,\qquad
\del*_\sigma\bar\varphi^{2}=0.
$$
Since, $\del\bar\varphi^1=0$ and $\del\bar\varphi^2=0$, we have
that
$$
\int_{\mathbb{I}}\del\eta\wedge\omega^{2}=0,
$$
for every $\eta\in\mathcal{H}^{0,1}_{\delbar}$, namely $\omega_\sigma$ satisfies the integral condition  (\ref{eq:ci}).
\end{ex}

The previous example shows that the integral condition on compact complex manifolds cannot be characterized by the strongly Gauduchon one so this leads to the following open question.
\begin{question**}
Is there a geometric condition on a Hermitian metric  that the integral condition characterizes?
\end{question**}

\section{The integral condition on almost complex surfaces}\label{sect:ci-surf}

We recall that on compact complex surfaces if there exists a strongly Gauduchon metric then there exists a K\"ahler metric by \cite{popovici-surface}. We wonder whether a similar result holds for the integral condition (\ref{eq:ci}), i.e. if   a compact complex surface $(M, J)$   endowed with a Hermitian metric satisfying the integral condition  admits  also a K\"ahler metric  (see Question \ref{conj:surface}). 
If we do not assume that $J$ is integrable, the answer to Question \ref{conj:surface} is negative. 

Namely, we will provide an explicit Hermitian metric on a compact almost complex surface without any complex structure and any compatible  almost K\"ahler metric. To this purpose, let $G=\hbox{\rm Nil}^4$ be the simply-connected and connected  $3$-step nilpotent Lie group whose Lie algebra $\mathfrak{g}$ has a basis $\{e_1,\ldots,e_4\}$ satisfying
$$
[e_1,e_2]=e_3,\qquad [e_1,e_3]=e_4\,.
$$
Accordingly, denoting by $\{e^1,\ldots,e^4\}$ the dual basis of $\{e_1,\ldots,e_4\}$, the following structure equations hold
$$
de^1=0,\quad de^2=0,\quad de^3=-e^{12},\quad de^4=-e^{13},
$$  
where $e^{ij}=e^i\wedge e^j$ and so on. Then $G$ admits lattices $\Gamma$ so that $M=\Gamma\backslash G$ is a compact $4$-dimensional nilmanifold. Since $b^1(M)=2$, $M$ has no complex structures.
Let us define the following almost complex structure $J$ given by
$$
Je_1=e_2,\qquad Je_2=-e_2,\qquad Je_3=e_4,\qquad Je_4=-e_3\,.
$$
Setting $\varphi^1$ and $\varphi^2$ for the associated global co-frame of $(1,0)$-forms, we obtain the following structure equations (see for instance \cite{cirici-wilson-2, tardini-tomassini})
$$
\left\lbrace
\begin{array}{lcl}
d\varphi^1 & =& 0\\
d\varphi^2 &=& \frac{1}{2i}\varphi^{12}+\frac{1}{2i}\left(\varphi^{1\bar 2}-\varphi^{2\bar 1}\right)-i\varphi^{1\bar 1}+\frac{1}{2i}\varphi^{\bar 1\bar 2}\,.
\end{array}
\right.
$$
In particular, notice that $J$ is non-integrable.
We fix the diagonal metric 
$$
\omega:=\frac{1}{2i}\left(\varphi^{1\bar 1}+\varphi^{2\bar 2}\right).
$$ 
Notice that $\omega$ is  invariant hence it is Gauduchon, i.e., $\del\delbar\omega=0$. Now, by a direct computation
$$
\del\omega=-\frac{1}{2}\varphi^{12\bar1}
$$
and $$\delbar\varphi^{12}=\frac{1}{2i}\varphi^{12\bar 1}.$$ Hence
$$
\del\omega=\delbar\left(-i\varphi^{12}\right).
$$
Therefore, by definition $\omega$ is strongly Gauduchon and so it satisfies the integral condition in view of Theorem  \ref{prop:sg-ci}.\\
One can prove directly that $J$ does not admit any left-invariant  compatible symplectic structure or we can refer to \cite{cirici-wilson-2} where a different argument is used.
Hence, by an average procedure, since $J$ is a left-invariant almost complex structure  on $G$ that does not admit any left-invariant compatible symplectic structure, then the induced invariant almost complex structure  $J$ on $M$  does not admit any compatible symplectic structure (cf. \cite[Proposition 3.9]{piovani}). So, $(M,J)$ is a $4$-dimensional almost complex manifold that is not almost K\"ahler but it admits a strongly Gauduchon metric that, in particular, satisfies the integral condition.\\

Summing up, we proved the following
\begin{theorem}
There exists a compact almost complex $4$-dimensional manifold $(M,J)$ that admits a strongly Gauduchon metric but it does not have any  compatible almost-K\"ahler metric. In particular, such a metric satisfies the integral condition (\ref{eq:ci}).
\end{theorem}

Notice that, in particular, the previous example shows that the result in \cite{popovici-surface} saying that strongly Gauduchon compact complex surfaces are K\"ahler does not hold as soon as we drop the integrability assumption on the complex structure.\\

As mentioned, it is still an open problem to understand the existence of Hermitian metrics satisfying the integral condition on non-K\"ahler surfaces, leading to the following
 \begin{question*}
Let $(M,J)$ be a compact complex surface. If there exists a Hermitian metric satisfying the integral condition then does $(M, J)$  admit a K\"ahler metric?
\end{question*}

\end{document}